\documentclass{amsart}
\usepackage{hyperref}
\usepackage{tikz}
\usepackage{cite}
\usetikzlibrary{arrows, cd}
\tikzset{> = stealth}
\tikzcdset{arrow style = tikz}

\newtheorem{thm}{Theorem}[section]
\newtheorem{prp}[thm]{Proposition}
\newtheorem{lem}[thm]{Lemma}
\newtheorem{cor}[thm]{Corollary}
\theoremstyle{change}
\theoremstyle{nonumberplain}

\theoremstyle{definition}
\newtheorem{dfn}[thm]{Definition}

\theoremstyle{remark}

\numberwithin{equation}{section}

\begin{document}

\title[Stable embedding calculus and configuration spaces]{The stable embedding tower and operadic structures on configuration spaces}


\author{Connor Malin}
\address{University of Notre Dame}
\email{cmalin@nd.edu}

\begin{abstract}
    Given smooth manifolds $M$ and $N$, manifold calculus studies the space of embeddings $\operatorname{Emb}(M,N)$ via the ``embedding tower'', which is constructed using the homotopy theory of presheaves on $M$.  The same theory allows us to study the stable homotopy type of $\operatorname{Emb}(M,N)$ via the ``stable embedding tower''. By analyzing cubes of framed configuration spaces, we prove that the layers of the stable embedding tower are tangential homotopy invariants of $N$.
    
    If $M$ is framed, the moduli space of disks $E_M$ is intimately connected to both the stable and unstable embedding towers through the $E_n$ operad. The action of $E_n$ on $E_M$ induces an action of the Poisson operad $\mathrm{pois}_n$ on the homology of configuration spaces $H_*(F(M,-))$. 
    In order to study this action, we introduce the notion of Poincaré-Koszul operads and modules and show that $E_n$ and $E_M$ are examples. As an application, we compute the induced action of the Lie operad on $H_*(F(M,-))$ and show it is a homotopy invariant of $M^+$.
\end{abstract}
\maketitle

\section{Introduction}

 For fixed smooth manifolds $M,N$, the philosophy of manifold calculus is to approximate $\mathrm{Emb}(M,N)$ by looking at how $\mathrm{Emb}(-,N)$ behaves on the poset of open sets of $M$ which are diffeomorphic to $\sqcup_i \mathbb{R}^n$ as $i$ ranges from $0$ to $\infty$. Weiss used this approach to construct a tower of approximations, $T_i(F)$, for a ``good'' presheaf $f:\mathrm{Open}(M)^{\mathrm{op}} \rightarrow \mathrm{Top}_*$\cite{weiss_1999}.

Much is known about the convergence of this tower for the presheaf $\mathrm{Emb}(-,N)$, in particular it converges, i.e. $\mathrm{Emb}(M,N) \simeq T_{\infty}(\mathrm{Emb}(M,N))$, if the codimension of $M$ and $N$ is at least 3\cite[Corollary 4.2.4]{goodwillie_klein_weiss_2001}, and in the case $\mathrm{Dim}(M)=\mathrm{Dim}(N)=2$\cite{kupers_krannich_surfaces}. In the case $M=S^1$ and $\mathrm{Dim}(N)=4$,  the convergence of the tower was used to show the homotopy type of $\mathrm{Emb}(S^1,M)$, the space of smooth knots in $M$, is a homeomorphism invariant of $N$, if $N$ is simply connected\cite[Theorem A]{knudsen-kupers}.  Recently, it was shown that for compact spin manifolds of dimension at least five, the tower for $\mathrm{Emb}(M,M)=\mathrm{Diff}(M)$ does not converge\cite[Theorem C]{krannich_kupers_structure}. 

The tower provides interesting information even when it does not converge. Extensive calculation has been done in the case $M=S^1,N=\mathbb{R}^3$; in particular, many classical invariants of knots can be constructed via pulling back along the approximations $\mathrm{Emb}(S^1 ,\mathbb{R}^3) \rightarrow T_i(\mathrm{Emb}(S^1 ,\mathbb{R}^3))$\cite{kosanovic}. 

So what does $T_{\infty}(\mathrm{Emb}(M,N))$ describe in the cases it does not converge? Boavida de Brito and Weiss show that in the case of ``context free functors'' like $ \mathrm{Emb}(-,N)$ what the tower naturally computes is the derived mapping space of modules over the framed little disks operad $\mathrm{Map}^h_{ E^{\mathrm{fr}}_n}
     ( E^{\mathrm{fr}}_M, E^{\mathrm{fr}}_N)$\cite[Proposition 8.3]{brito_weiss_2013}. Here $E^{\mathrm{fr}}(M)$ denotes $\mathrm{Emb}(\sqcup_{i \in I}\mathbb{R}^n, M)$ as $I$ ranges over all finite sets. The $E^{\mathrm{fr}}_n$ operad acts upon this by precomposition. Moreover, the tower arises by considering mapping spaces of modules over truncations of the framed disks operad. In particular, the tower only depends on the module structures.

Kupers--Randall-Williams made complete computations of the rational homotopy groups of $B\mathrm{Diff}_\partial(D^{2d},D^{2d})$ outside of certain ``bands'' of degrees \cite{kupers_randal}. Here $B\mathrm{Diff}_\partial(D^{2d},D^{2d})$ denotes the diffeomorphisms of a disk relative the boundary. These calculations employ an alternative theory of manifold calculus for framed manifolds which is designed to study the space of framed embeddings, via derived maps of $E_n$-modules $\mathrm{Map}^h_{E_n}(E_M,E_N)$. The modules $E_M$ and $E_N$ are given by framed embeddings of disks into $M$. The addition of the framing information implies that $E_M \simeq F(M,-)$, the collection of all ordered configurations in $M$.
 
 Applying homology to $E_M$ gives the homology of configuration spaces of a framed manifold the algebraic structure of a $\mathrm{pois}_n$ module via the operad isomorphism $H_*(E_n )\cong \mathrm{pois}_n$\cite{sinha}. In particular, this algebraic structure is an invariant of the framed diffeomorphism type of $M$. In practical terms, this means that given an element $x \in H_*(F(M,I))$, an element $a \in I$, and a forest of binary trees with leaves labeled by $J$, we produce an element $y \in H_*(F(M, I-\{a\} \sqcup J))$ with some shift of homological degree. Moreover, these operations satisfy equivariance, associativity, Jacobi identities, and a derivation formula relating the action of a forest to the action of its trees.

In the first half of this paper, we begin to investigate the module structure of $E_M$. We define notions of Poincaré-Koszul operads and Poincaré-Koszul modules and observe the homology of these naturally comes with a Koszul self duality map. We use an observation of Salvatore\cite{salvatore_2021} that the Fulton-MacPherson compactifications $\mathcal{F}_n,\mathcal{F}_M$ are naturally isomorphic to certain cofibrant replacements to deduce these operads and modules are Poincaré-Koszul.

\begin{thm}
For a tame, framed $n$-manifold $M$ without boundary, there are isomorphisms of modules 
\[H_*(E_M) \cong s_{(n,n)}H_*(K(E_{M^+}))\]
compatible with the isomorphism of operads
\[H_*(E_n) \cong s_n H_*(K(E_n)).\]

\end{thm}

As an application we show,

\begin{cor}
For a tame, framed $n$-manifold $M$, the shifted Lie module $\mathrm{res}_{\mathrm{s_n lie}}H_*(E_{M}) $ is a homotopy invariant of $M^+$.
\end{cor}

This adds to a growing list of results about the interaction of higher algebraic structures and configuration spaces. Campos--Willwacher show a characteristic 0 cochain version of this statement which takes into account the full action of the $E_n$ operad \cite[Theorem 1]{campos2022model}. In a different vein, Petersen shows various homotopy invariance results for the compactly supported cohomology of configuration spaces. In particular, $H_c^*(F(X,k);\mathbb{Z})$ depends only $C_c^*(X;\mathbb{Z})$ as an $E_\infty$-algebra \cite[Corollary 5.16]{Petersen_2019}.

In the second half of this paper, we study the homotopy invariance of manifold calculus. In the case of the presheaf $\mathrm{Emb}(-,N)$, or equivalently the module $E^{\mathrm{fr}}_N$, a natural question is how the associated tower behaves under homotopy equivalence of $N$. From now on we assume $M,N$ are diffemorphic to the interiors of the compact manifolds with (possibly empty) boundary $\bar{M}$ and $\bar{N}$.

The $i$th layer of the tower, defined as \[D_i \mathrm{Emb}(-,N):=\mathrm{hofib}(T_i(\mathrm{Emb}(-,N))\rightarrow T_{i-1}(\mathrm{Emb}(-,N))),\] is given in terms of section spaces where the fibers are determined by cubes of framed configuration spaces of $N$ (if $i>2$ the fibers are also determined by the unframed configuration spaces \cite[Section 9]{weiss_1999}). It is well known that in general $F(N,k)$ is not a homotopy invariant of N\cite{longoni_salvatore_2005}, so a straightforward equivalence of the layers, let alone the towers, is out of the question.

There are, however, results about the homotopy invariance of the suspension spectra of configuration spaces. Aouina and Klein gave the original nonequivariant answer in the case of a PL manifold without boundary\cite[Theorem A]{aouina_klein_2004}, and recently Knudsen and the author separately gave equivariant results for tame topological manifolds\cite[Theorem C]{knudsen_2018}\cite[Theorem 1.1]{malin}. One might conjecture that this invariance should have consequences in stable embedding calculus, i.e. calculus applied to the functor $\Omega^\infty (\Sigma^\infty_+ \mathrm{Emb}(-,N) \wedge E)$ for some spectrum $E$.

 While less is known about the stable embedding tower, if we assume the same codimension restrictions, one can expect convergence of the stable embedding tower if one manually kills off the first few terms in the tower\cite[Theorem 2.2]{weiss_2004}. Arone--Lambrechts--Volic have studied the stable embedding tower for \[\Omega^\infty (\mathrm{hofib}(\mathrm{Emb}(M,\mathbb{R}^d) \rightarrow \mathrm{Imm}(M,\mathbb{R}^d))\wedge H\mathbb{Q})\] and used it to show that the rational homology of $(\mathrm{hofib}(\mathrm{Emb}(M,\mathbb{R}^d) \rightarrow \mathrm{Imm}(M,\mathbb{R}^d))$ for large codimension depends only on the rational homology type of $M$\cite[Theorem 1.8]{arone_lambrechts_volic_2007}.
 
 The goal of the second half of this paper is to show that the layers of the embedding tower for any space valued presheaf of the form $F \circ \Sigma^\infty_+ \mathrm{Emb}(-,N)$ are a tangential homotopy invariant of $(\bar{N},\partial \bar{N})$, which, in particular, implies it is a proper tangential homotopy invariant of $N$. A similar result for the stable orthogonal tower was proven for parallelizable manifolds by Arone\cite[Corollary 1.13]{arone_2009}.

We approach the problem of the tangential invariance of the layers by generalizing the invariance of stable configurations in two ways. First, we allow for framings. More precisely, if $\xi \rightarrow N$ is a bundle with compact manifold fibers (possibly with boundary), then $F^\xi(N,k)$ is defined to be $\xi^k|_{F(N,k)}$. We show that $\Sigma^\infty_+ F^\xi(N,k)$ is a relative fiberwise homotopy invariant of $\xi$. Second, we show that the cube $J \subset I$, $J \rightarrow \Sigma^\infty_+ F^\xi(N,J)$, with maps given by projection, is a relative fiberwise homotopy invariant of $\xi$. Though there is a formidable amount of notation and bookkeeping, the heart of this argument is a combination of two definitions: Definition \ref{def:thick} which allows us to use duality to pass to quotients of bundles and Definition \ref{def:cyl} which extends these quotients along a mapping cylinder.

Applying this analysis to the orthonormal frame bundle, combined with the aforementioned description of the layers in terms of framed configuration spaces yields:

\begin{thm}
For a smooth, tame $n$-manifold $N$ and functor $F:\mathrm{Sp} \rightarrow \mathrm{Top}_*$ that preserves weak equivalences, the layers of the tower for $F \circ \Sigma^\infty_+ \mathrm{Emb}(-,N)$ are relative tangential homotopy invariants of $(\bar{N},\partial \bar{N})$.
\end{thm}

\section{(co)Operads and (co)modules}
While operads have an incredibly rich and complex homotopy theory, they are often accessible via geometric means. The interaction of operads and configuration spaces was noticed by Cohen, who, in calculating the homology of the configuration spaces of $\mathbb{R}^n$, calculated the homology of the little disks operad $E_n$\cite{cohen_lada_may_1976}. For an account of this calculation which emphasizes the relation with the Poisson operad $\mathrm{pois}_n$, see\cite{sinha}. 

There are many equivalent ways to define a (co)operad each with its own advantages. The definitions we lay out here, in terms of partial (de)composites, are particularly suited for our geometric perspective. The starting point for most of these definitions are symmetric sequences in a symmetric monoidal category $(C,\otimes)$. For us $(C,\otimes)$ will usually be one of $(\mathrm{Top},\times),(\mathrm{Top}_*,\wedge),(\mathrm{Sp},\wedge),(\mathrm{dgVect}_k,\otimes)$. By $(\mathrm{Top},\times), (\mathrm{Top}_*,\wedge)$ we mean a convenient category of topological spaces and by $(\mathrm{Sp},\wedge)$ we mean a symmetric monoidal category of spectra. 

\begin{dfn}
A symmetric sequence in $C$ is a functor from the category of nonempty finite sets and bijections to $C$.
\end{dfn}

Given finite sets $I,J$ with $a \in I$, we define  the infinitesimal composite $I \cup_a J:=I -\{a\} \sqcup J$. The combinatorics of infinitesimal composites will dictate the axioms of an operad.

\begin{dfn}
An operad in $(C,\otimes)$ is a symmetric sequence $O$ in $C$ together with partial composites:
\[O(I) \otimes O(J) \rightarrow O(I \cup_a J)\]
for all $a \in I$. These satisfy straightforward equivariance, unital, and associativity conditions.
\end{dfn}

\begin{dfn}
A cooperad in $(C,\otimes)$ is a symmetric sequence $P$ in $C$ together with partial decomposites:
\[P(I \cup_a J) \rightarrow P(I) \otimes P(J)\]
for all $a \in I$. These satisfy straightforward equivariance, counital and  coassociativity conditions.
\end{dfn}

For a thorough introduction to (co)operads using these definitions, we refer to\cite[Section 2]{ching_2005}. From now on we restrict to reduced (co)operads, i.e. those cooperads whose underlying symmetric sequence $S$ has  $S(\{*\})$ equal to the unit of $\otimes$. Any operad $O$ in $(\mathrm{Top},\times)$ with $O(\{*\})\simeq *$ is equivalent to a reduced operad by replacing $O(\{*\})$ by the singleton set containing its unit. For example, the $E_n$ operad is not reduced, but $E_n(\{*\})\simeq *$ with a canonical basepoint given by the identity. When we refer to $E_n$, we implicitly make this reduction.

\begin{dfn}
A right module $R$ over an operad $O$ in $(C,\otimes)$ is a symmetric sequence $R$ in $C$ with partial composites:
\[R(I) \otimes O(J) \rightarrow R(I \cup_a J)\]
for all $a \in I$. These satisfy straightforward equivariance, unital, and associativity conditions.

\end{dfn}

\begin{dfn}
A right comodule $R$ over a cooperad $P$ in $(C,\otimes)$ is a symmetric sequence $R$ in $C$ together with partial decomposites:

\[R(I \cup_a J) \rightarrow R(I) \otimes P(J)\]
for all $a \in I$. These satisfy straightforward equivariance, counital, and coassociativity conditions.
\end{dfn}

\section{The module structure on homology}

\label{sec:algduality}
The homology of configuration spaces is intertwined with the Lie operad in two distinct ways. Knudsen shows that the rational homology of unordered configuration spaces can be calculated as the Lie algebra homology of a certain Lie algebra associated to the manifold\cite[Theorem 1.1]{knudsen_2017}. However, there is also a more direct connection between the homology of configuration spaces and the Lie operad which we now describe.

For a fixed field $k$, consider the category $(\mathrm{dgVect}_k,\otimes)$ of chain complexes. We have a few notable operads: $\mathrm{lie},\mathrm{com}$ and $\mathrm{pois}_n$, all endowed with the trivial differential. The first governs graded Lie algebras with a bracket of degree $-1$, the middle governs graded commutative algebras, and the latter governs n-Poisson algebras which have a Lie bracket of degree $n-1$ and a graded commutative product for which the Lie bracket is a derivation. There are maps of operads $\mathrm{lie} \rightarrow s_{-n} \mathrm{pois}_n \rightarrow s_{-n} \mathrm{com}$, where $s_{-n}$ is a shifting operation defined in Section \ref{sec:duality} The first encodes forgetting the commutative product and the latter encodes adding a trivial Lie bracket to a commutative algebra. Classically, it is known if the characteristic of $k$ is 0, there is an equivalence between suitably connected categories of cocommutative coalgebras and Lie algebras. This is an instance of Koszul duality, which underlies an operadic Koszul duality functor $\mathrm{Operad}((C,\otimes)) \xrightarrow{K} \mathrm{Operad}((C,\otimes))$ where $(C,\otimes)=(\mathrm{dgVect_k},\otimes)$. For a thorough discussion of Koszul duality for algebras and operads, we refer to\cite{loday_vallette_2012}. There is a remarkable fact due to Getzler-Jones that $H_*(K(\mathrm{lie})) \leftarrow H_*(K(s_{-n}\mathrm{pois}_n)) \leftarrow H_*(K(s_{-n}\mathrm{com}))$ is isomorphic to $\mathrm{com} \leftarrow \mathrm{pois}_n \leftarrow s_n \mathrm{lie}$\cite{getzler_jones}.

As mentioned before, $H_*(E_n)\cong\mathrm{pois_n}$. Hence, if $R$ is a module over $E_n$, then $H_*(R)$ is a module over $\mathrm{pois}_n$. We would like to construct such a module where $R(I) \simeq F(M,I)$, the $I$-labeled configurations of $M$. The quickest way to define such a module is as follows:

For framed $n$-manifolds $M,N$ define $\mathrm{Emb}^{\mathrm{fr}}(M,N)$ as the space of embeddings from $M$ to $N$ with the information an isotopy of the framing of $M$ to agree with its image \cite[Definition 2.8]{malin2023koszul}. Ayala-Francis show the symmetric sequence $\mathrm{Emb}^{\mathrm{fr}}(\sqcup_{i \in I}\mathbb{R}^n,\mathbb{R}^n)$ forms an operad equivalent to $E_n$\cite[Remark 2.10]{ayala_francis_2015}. Then for any framed manifold $M$, the symmetric sequence $\mathrm{Emb}^{\mathrm{fr}}(\sqcup_{i \in I}\mathbb{R}^n,M)$ is a right module by precomposition of framed embeddings, which as a symmetric sequence is equivalent to $F(M,I)$. For convenience, we refer to these modules as $E_M$. Applying homology endows $H_*(F(M,-))$ with structure of a $\mathrm{pois}_n$ module. 

We let $E_{M^+}$ denote the quotient of $E_M$ which collapses those configurations of disks for which one is mapped into $\infty$ \cite[Definition 2.14]{malin2023koszul}. Again, by radial contraction, the $I$th space of this module has the homotopy type of \[\{(x_i) \in (M^+)^{\wedge I}| x_i = x_j \Rightarrow i = j\} \cup \{(*,*,\dots,*)\}.\]


This first half of the paper concerns the $\mathrm{pois}_n$ module structure on $H_*(F(M,-))$. In particular, we show that $\mathrm{res}_{s_n\mathrm{lie}}H_*(F(M,-))$, the module restriction along the morphism $s_n\mathrm{ lie}\rightarrow \mathrm{pois}_n$, is a homotopy invariant of $M^+$, which is somewhat surprising as the construction of this action makes heavy use of the locally Euclidean nature of $M$. 

\section{Operad-Cooperad duality}

In this section, we explain how the notion of the Koszul dual cooperad, due to Ching, Salvatore\cite[Section 4.1]{ching_2005}\cite{salvatore_1998}, factors as a composition of a homotopical construction followed by one point compactification. Recall that one point compactifications $(-)^+$ of locally compact spaces exist and are covariant with respect to proper maps and contravariant with respect to open inclusions. From now on, all (co)operads are reduced.

\begin{dfn}
An open operad $O$ is an operad in $(\mathrm{Top},\times)$ such that all partial compositions $O(I) \times O(J) \rightarrow O(I \cup_a J)$ are open embeddings.
\end{dfn}

For example,  Arone-Kankaanrinta  construct operad structures on the symmetric sequences $(\mathbb{R}^n)^I /\{\mathrm{translation}\}$ where the partial composites are homeomorphisms, hence open embeddings. Taking one point compactifications then yields ``sphere operads'' which are further examples of open operads\cite{arone_kankaanrinta_2013}. However, in general, the one point compactification of an operad will not be an operad.

For a symmetric sequence $S$ in $(\mathrm{Top},\times)$, we let $S^+$ denote the symmetric sequence in $(\mathrm{Top}_*,\wedge)$ given by $S^+(I)=S(I)^+$.

\begin{prp}
Let $O$ be an open operad. Then $O^+$ is naturally a cooperad in pointed spaces.
\end{prp}

\begin{proof}
One point compactification is contravariant with respect to open embeddings, hence we have maps $O(I \cup_a J)^+ \rightarrow (O(I) \times O(J))^+ = O(I)^+ \wedge O(J)^+$. 
\end{proof}

We recall the definitions of the W-construction and bar construction for a reduced operad $O$ in $(\mathrm{Top}_*,\wedge)$ in terms of trees. For a quick treatment, see\cite[Section 6]{ching_salvatore}.

The following is originally due to Boardman-Vogt\cite{boardman_vogt_1973}:
\begin{dfn}
Given an operad O in $(\mathrm{Top}_*,\wedge)$, let $T(O)$ denote the space
of rooted trees with the property that the root has a single adjacent edge and every internal vertex has at least 2 children with labels as follows: the internal vertices are labeled by $O$ while nonroot and nonleaf adjacent edges are
labeled by an element of $[0, \infty]$; any tree with a vertex labeled by the basepoint is collapsed to a single point. We let $W(O)$ denotes the quotient of $T(O)$ by the relation that any length 0 edge can be collapsed
by applying operadic partial composition.
$W(O)$ is an operad by grafting trees via length $\infty$ edges.
\end{dfn}

In the case $O$ is an operad in unpointed spaces, we will let $W(O)$ denote the operad in unpointed spaces given by $W(O_+)(I)-\{*\}$.  It is well known that operad composition induces an equivalence of operads $W(O) \rightarrow O$.

\begin{dfn}
Given an operad $O$ in $(\mathrm{Top}_*,\wedge)$, let $B(O)$ denote $W(O)\wedge
(0, \infty)^+$, where the $(0, \infty)$ coordinate is interpreted a labeling of the edge adjacent to the root, modulo the relations that any tree with an $\infty$ length edge is identified with
the basepoint.

This is a cooperad via returning the unique decomposition of trees compatible with the decomposition $I \cup_a J$ (or sending to the basepoint if no such decompostion exists).

\end{dfn}

If $O$ is an operad in $(\mathrm{Top},\times)$ we let $B(O)$ denote $B(O_+)$. An important difference in how we define $W(O)$ and $B(O)$ for an operad in unpointed spaces is how we treat basepoints. $W(O)$ is an operad in $(\mathrm{Top},\times)$ while $B(O)$ is a cooperad in $(\mathrm{Top}_*,\wedge)$.

For the rest of this section, we fix an operad $O$ in $(\mathrm{Top},\times)$.

\begin{dfn}
Let $\partial W(O)$ represent the subspace of $W(O)$ with some edge of length $\infty$. We let $\mathring{W}(O):=W(O)-\partial W(O)$
\end{dfn}

It is easy, but not necessary, to check that $\partial W(O)$ is a model of the derived decomposables of $O$.

\begin{dfn}
Let $\mathring{W}_{(0,\infty)}(O):=\mathring{W}(O) \times (0,\infty)$. It is an operad via grafting trees by an edge determined by the second trees $(0,\infty)$ coordinate. Similarly, define $W_{[0,\infty]}:=W(O) \times [0,\infty]$ by grafting via the $[0,\infty]$ coordinate.
\end{dfn}

As before, we think of the $(0,\infty)$ or $[0,\infty]$ factor as the weight of the root edge.

\begin{dfn}
An operad $O$ is compact if it is made up of compact spaces.
\end{dfn}

\begin{prp}\label{prp:bar}
 $\mathring{W}_{(0,\infty)}(O)$ is an open operad, and if $O$ is compact then \[\mathring{W}_{(0,\infty)}(O)^+ \cong B(O).\]
\end{prp}

\begin{proof}
The composition maps are clearly embeddings since if a tree can be written via a specific partial composition then the factors are unique. This relies on the requirement that all internal vertices have more than one child, which is also why we require $O(1)=*$. The fact that the subset is open is also clear: the ways to perturb a tree in $\mathring{W}_{(0,\infty)}(O)(I \cup_a J)$ are by either perturbing the vertices or by perturbing the edge length. But the exact same is true in $\mathring{W}_{(0,\infty)}(O)(I ) \times \mathring{W}_{(0,\infty)}(O)(J )$

The second statement on the level of spaces is clear: the bar construction consists of weighted (including the root edge) trees  modulo the trees with some edge $\infty$ or the root edge $0$. In our case, since $O$ is compact, the noncompactness of $\mathring{W} (O)\times (0,\infty)$ is due to the lack of length $\infty$ edges and the lack of length $0$ root edge. If we one point compactify, we will see that that as a space we get exactly the bar construction. That the cooperad partial decomposition maps are equal follows simply from definitions.
\end{proof}

The bar construction exists in far more generality than the above. In particular, it exists for operads in $(\mathrm{Sp},\wedge)$ and results in a cooperad in $(\mathrm{Sp},\wedge)$. We refer to\cite[Section 4]{ching_2005} for the details of the construction, but we will only need that for an operad in unpointed spaces $O$, $\Sigma^\infty B(O)\simeq B(\Sigma^\infty_+ O)$. 

Recall that for a spectrum $X$, the Spanier--Whitehead dual $X^\vee$ is the spectrum of functions $F(X,S^0)$. It is straightforward to check that for a cooperad $P$ in spectra $P^\vee(I):=P(I)^\vee$ is naturally an operad. Unfortunately, the reverse is not true. The dual of an operad is not a cooperad, in general.

\begin{dfn}
For an operad $O$ in $(\mathrm{Sp},\wedge)$, the Koszul dual operad is \[K(O):= B(O)^\vee.\]
If $O$ is an operad in unpointed spaces, we will refer to $K(\Sigma^\infty_+ O)$ as $K(O)$.\end{dfn}

 Ching shows that this topological Koszul duality is a lift of classical Koszul duality, in the sense that there is a spectral sequence \[ K(H_*(O)) \Longrightarrow H_*(K(O)).\] This spectral sequence collapses if $H_*(O)$ is Koszul\cite[Proposition 9.39, 9.48]{ching_2005}. This is a commonly occuring property of algebraic operads see\cite[Chapter 7]{loday_vallette_2012}. In particular, it collapses for $E_n$ since $\mathrm{pois}_n$ is Koszul.

\section{Poincaré-Koszul operads}\label{sec:duality}

In this section, we define what it means for a topological operad to be Poincaré-Koszul, with respect to a fixed field $k$. All homology is taken with respect to $k$. In particular, it should imply an equivalence $H_*(O) \simeq s_n H_*(K(O))$. Informally, the pair of $O$ and its derived decomposables should be a Poincaré duality pair (see Section \ref{sec:poincare}). The main goal of this section is to show that the $E_n$ operad is Poincaré-Koszul. In the next section we prove analogous results for the modules $E_M$ associated to a framed manifold.

For the remainder of this section, fix a reduced operad $O$ in unpointed spaces such that for all $I$, $O(I)$ is homotopy equivalent to a finite complex. 

\begin{dfn}
The symmetric sequence $\partial (W_{[0,\infty]}(O))$ is given by
\[((\partial W(O)(I))\times [0,\infty]) \cup (W(O)(I) \times \{0\}) \cup (W(O) (I)\times \{\infty\})\]
\end{dfn}

\begin{dfn}
A distinguished $n$-class of $O$ is a homology class $\alpha_I$, for each nonempty finite set $I$, in the relative homology \[H_{n|I|-n}(W_{[0,\infty]}(O)(I),\partial(W_{[0,\infty]}(O))(I))\] such that, under the isomorphism with $\bar{H}_*(B(O))$, the partial decomposites act as
\[\alpha_{I \cup_a J} \rightarrow \alpha_I \otimes \alpha_J.\]

\end{dfn}

As a sanity check, note that \[(n|I|-n) +( n|J|-n)=n(|I|+|J|-1)-n=n|I \cup_a J|-n.\]

\begin{dfn}
An operad $O$ with a distinguished $n$-class $\alpha$ is Poincaré-Koszul of dimension $n$ if $\alpha_I$ makes \[(W_{[0,\infty]}(O)(I),\partial(W_{[0,\infty]}(O))(I))\] into a Poincaré duality pair for all $I$. We refer to $\alpha$ as the fundamental class of the operad.
\end{dfn}
There is a natural notion of suspension for operads in the category $(\mathrm{dgVect},\otimes)$. For a chain complex $A$, let $A[i]$ denotes the graded vector space where the gradings have been shifted up $i$.

\begin{dfn}
The operad $S_n$ in $(\mathrm{dgVect}_k,\otimes)$ is defined by \[S_n(I):=k[n|I|-n]\] with partial composites determined by the canonical isomorphism $k \otimes k \cong k$.
\end{dfn}

We define the $n$th suspension of an operad $O$ in $(\mathrm{dgVect},\otimes)$ by \[s_n O(I):= S_n(I) \otimes O(I).\] This naturally forms an operad. 
\newpage
By unfolding definitions, one immediately obtains:

\begin{thm}
If $O$ is a Poincaré-Koszul operad of dimension $n$, then there is an isomorphism of operads

\[H_*(O) \cong s_n H_*(K(O)).\]

induced by \[H^*(W_{[0,\infty]}(O)) \xrightarrow{\cap \alpha} s_{-n}\bar{H}_*(B(O)),\] where we consider $H^i(-)$ as living in degree $-i$. This isomorphism is natural with respect to maps of Poincaré-Koszul operads of dimension $n$ that preserve the fundamental $n$-class

\end{thm}

A less important, but still relevant result, is that if we consider the suboperad of decomposables $X:=\mathrm{decom}(W_{[0,\infty]}(O))$, the quotient of $X$ by the trees with an edge of length infinity or the root edge with length 0, naturally forms a cooperad, again by degrafting. Let's call this quotient $Y$.
\begin{prp}\label{prp:extra}
There is an isomorphism of operads $H_*(X)\cong s_n H_*(Y^\vee)$.
\end{prp}

 This duality is a reflection of the fact that $\partial W(O)$ is a Poincaré duality space. The cooperad $X^+$ should be thought of as a type of ``dendroidal suspension" of the derived decomposables. This $Y$ receives a degree 1 collapse map from $B(O)$, similar to how, for a compact manifold $N$, $\Sigma \partial N$ receives a degree $1$ map from $N/\partial N$.

Recall that a manifold is tame if it is homeomorphic to the interior of a compact manifold with boundary. Tameness allows us to make the identifications $H_*(\Sigma^\infty_+ M^\vee) \cong H^*(M)$ and $H_*(\Sigma^\infty (M^+)^\vee) \cong \bar{H}^*(M^+)$ since it implies $M$ and $M^+$ are both homotopy finite.

\begin{dfn}
A manifold operad $O$ is an open operad for which all $O(I)$ are tame topological manifolds without boundary. We say a manifold operad is oriented (with respect to a field $k$) if we are supplied with local orientations of the $O(I)$ such that the partial composites are orientation preserving.

\end{dfn}

\begin{dfn}
The dimension of a manifold operad $O$ is $\operatorname{dim}(O(J))$, $|J|=2$.
\end{dfn}As mentioned earlier, Arone-Kankaanrinta supply an infinite family of nonisomorphic operad structures on the symmetric sequences $(\mathbb{R}^n)^I/\{\mathrm{translation}\}$ and $(S^n)^{\wedge I}/\{\mathrm{translation}\}$ that yield important examples of manifold operads\cite{arone_kankaanrinta_2013}. These operads have dimension $n$.

%


%

%

\begin{thm}
Let $O$ be an oriented manifold operad of dimension $n$. Then the homological Poincaré duality isomorphism \[H_*(O) \cong s_n H_*((O^+)^\vee)\] is an isomorphism of operads.
\end{thm}

\begin{proof}
As before, one could appeal to the naturality of the (compactly supported) cap product. For variety, we supply an argument which suffices for manifold operads where the underlying spaces are homeomorphic to smooth manifolds.
\newpage
It suffices to show the following composite, defined via the geometric Poincaré duality pairing\cite{goresky_1981},

\begin{centering}
$H_*(O(I)) \otimes H_*(O(J))\otimes \bar{H}_*(O^+(I \cup_a J))$ \\
$\downarrow$\\
$H_*(O(I \cup_a J)) \otimes \bar{H}_*(O^+(I \cup_a J)) $\\
$\downarrow $\\
$S_n(I \cup_a J)$\\
\end{centering}
is equal to the composite

\begin{center}

$H_*(O(I)) \otimes H_*(O(J))\otimes \bar{H}_*(O^+(I \cup_a J)) $\\
$\downarrow$

 $H_*(O(I)) \otimes H_*(O(J)) \otimes \bar{H}_*(O^+(I)) \otimes \bar{H}_*(O^+(J))$\\
 $\downarrow$ \\
 $S_n(I) \otimes S_n(J)$\\
 $\downarrow$
\\$S_n(I \cup_a J)$

\end{center}

Either composite can only be nontrivial if $|x|+|y|$ is complementary to $|z|$. In this case, the first composite can be expressed by taking a generic representative $(x,y,z) \in H_*(O(I)) \otimes H_*(O(J)) \otimes \bar{H}_*(O^+ (I \cup_a J))$ and pushing forward $(x \times y)$ via the operad composition maps and counting oriented intersections with $z$. The bottom composite can be described as pulling $z$ back and counting intersections with $x \times y$. Since the partial composition maps are codimension 0 embeddings, these are evidently the same.

\end{proof}

\begin{thm}[Poincaré-Koszul duality of the $E_n$ operad]
The operad $E_n$ is Poincaré-Koszul of dimension $n$. In particular, there is a canonical isomorphism of operads \[H_*(E_n) \cong s_n H_*( K(E_n)).\]
\end{thm}

\begin{proof}
Salvatore shows that the compact model of $E_n$ called $\mathcal{F}_n$, the Fulton-MacPherson operad, has the property that $W(\mathcal{F}_n)\cong\mathcal{F}_n$ as operads\cite{salvatore_2021}. 
This implies $\mathring{W}_{(0,\infty)}(\mathcal{F}_n)$ is a manifold operad since it is automatically an open operad, and the $I$th space is homeomorphic to $(\mathcal{F}_n-\partial \mathcal{F}_n) \times (0,\infty)$ which is a tame manifold. Orientability follows from the fact the boundary embeddings $\mathcal{F}_n(I) \times \mathcal{F}_n(J) \rightarrow \mathcal{F}_n(I \cup_a J)$ respect the canonical local orientations of the $\mathcal{F}_n(I)$, in the sense that $[\mathcal{F}_n(I)]\otimes [\mathcal{F}_n(J)]\rightarrow \partial[\mathcal{F}_n(I \cup_a J)]$. Hence, by Proposition \ref{prp:bar} and Poincaré duality for manifold operads, there is an isomorphism \[ H_*(\mathring{W}_{(0,\infty)}(\mathcal{F}_n)) \cong s_n H_*(K(\mathcal{F}_n)).\]

But of course, $H_*(\mathring{W}_{(0,\infty)}(\mathcal{F}_n)) \cong H_*(W_{[0,\infty]}(\mathcal{F}_n))$ and the homology groups $\bar{H}_*(B(\mathcal{F}_n))$ can also be computed as \[H_*((W_{[0,\infty]}(\mathcal{F}_n),\partial ((W_{[0,\infty]}(\mathcal{F}_n))).\]

Since Poincaré duality for manifolds is expressible either through compactly supported cap products or relative cap products, we deduce that, under these identifications, this isomorphism is actually a Poincaré-Koszul duality isomorphism. 
\end{proof}


\section{Module-Comodule duality}
The story of Koszul duality for right modules is largely parallel to operads, though there are a few differences when it comes to one point compactification. From now on, we refer to right (co)modules over a (co)operad by the term ``(co)module''. 

\begin{dfn}
A proper module $R$ over a compact operad $O$ in $(\mathrm{Top},\times)$ is a module such that all partial composites are proper maps.
\end{dfn}

\begin{prp}
If $R$ is a proper module over the compact operad $O$, $R^+$ is a module in $(\mathrm{Top}_*,\wedge)$ over $O_+$.
\end{prp}

\begin{dfn}
An open module $R$ over an open operad $O$ is a module in $(\mathrm{Top},\times) $ such that all partial composites are open embeddings.
\end{dfn}

\begin{prp}
If $R$ is an open module over the open operad $O$, $R^+$ is a comodule over $O^+$.
\end{prp}

Note that the one point compactification of a proper module is a module and the one point compactification of an open module is a comodule. We recall the definitions of the $W$-construction and bar construction of a right module, due to Ching \cite{ching_2005}.

\begin{dfn}.
Given a module $R$ in $(\mathrm{Top}_*,\wedge)$ over the operad $O$, let $T(R)$ denote the space
of rooted trees such that every internal vertex has at least 2 children with labels as follows: the root is labeled by $R$,  the internal vertices are labeled by $O$, while nonleaf adjacent edges are
labeled by an element of $[0, \infty]$; we identify any vertex labeled by a basepoint to a single point. We let $W(R)$ denotes the quotient of $T(R)$ by the relation that any length 0 edge can be collapsed
by applying module or operad partial composition.
$W(R)$ is a module over $W(O)$ by grafting via length $\infty$ edges and a module over $W_{[0,\infty]}(O)$ by grafting via the $[0,\infty]$ coordinate.
\end{dfn}

If $R$ is a module in $(\mathrm{Top},\times)$, let $W(R)$ be the module given by $W(R_+)(I)-\{*\}$. It is well known that partial composition induces an equivalence of modules $W(R) \rightarrow R$ compatible with the equivalence $W(O) \rightarrow O$.

\begin{dfn}
If $R$ is a module over $O$ in $(\mathrm{Top}_*,\wedge)$, let $B(R)$ denote $W(R)$  modulo the relations that any tree with an $\infty$ length edge is identified with
$*$.
This is a comodule via decomposing trees, if possible, and otherwise sending to the basepoint.

\end{dfn}

As a symmetric sequence, the bar construction of a operad $O$ in $(\mathrm{Top}_*,\wedge)$ is homeomorphic to the bar construction of a module: $B(1)$. Here $1$ is the unique nontrivial module over $O$ which consists of $S^0$ in cardinality $1$, and $*$ otherwise. If $R$ is unpointed, we let $B(R)$ denote $B(R_+)$. Just as in the operad case, for an unpointed module $R$, $W(R)$ lives in $(\mathrm{Top},\times)$ and $B(R)$ in $(\mathrm{Top}_*,\wedge)$.

For the rest of this section, we fix a module $R$ over $O$ in $(\mathrm{Top},\times)$.

\begin{dfn}
Let $\partial W(R)$ represent the subspace of $W(R)$ with some edge of length $\infty$. We let $\mathring{W}(R):=W(R)-\partial W(R)$. It is a module over $\mathring{W}_{(0,\infty)}(O)$ by grafting via the $(0,\infty)$ coordinate.
\end{dfn}

It is easy, but not necessary, to check that $\partial W(R)$ is a model of the derived decomposables of $R$.

\begin{prp}\label{prp:barM}
The module $\mathring{W}(R)$ is an open module over $\mathring{W}_{(0,\infty)}(O)$. If $O$ is compact and $R$ is proper, $\mathring{W}(R)^+\cong B(R^+)$.
\end{prp}

\begin{proof}
As before, this is an exercise in understanding why $\mathring{W}(R)$  is noncompact. Since $O$ is compact, it is not coming from the internal vertices, but only from the lack of length $\infty$ edges and the noncompactness of the labels of the root vertex. Thus, one point compactification results in $B(R^+)$.
\end{proof}

As in the case of operads, the bar construction exists for modules in $(\mathrm{Sp},\wedge)$. See\cite[Section 7]{ching_2005} for details. We only need that for a module $R$ in $(\mathrm{Top},\times)$, $\Sigma^\infty B(R) \simeq B(\Sigma^\infty_+ R).$

\begin{dfn}
For a module $R$ in $(\mathrm{Sp},\wedge)$ the Koszul dual module is \[K(R):=B(R)^\vee.\] If $R$ is a module in unpointed spaces, we will refer to $K(\Sigma^\infty_+ R)$ as $K(R)$.
\end{dfn}
\noindent Again there is a spectral sequence \cite[Proposition 9.39]{ching_2005}
\[K(H_*(R)) \Longrightarrow H_*(K(R)).\]

\section{Poincaré-Koszul modules}

In this section, we define what it means for a module over a Poincaré-Koszul operad to be Poincaré-Koszul with respect to a field $k$. In particular, it should imply an equivalence $H_*(R) \cong s_{(n,d)} H_*(K(R))$ where $s_{(n,d)}$ is a version of suspension for modules. Informally, the pair of $R$ and its derived decomposables should be a Poincaré duality pair (see Section \ref{sec:poincare}).

For the remainder of this section, fix an operad $O$ and a right module $R$ in unpointed spaces such that for all $I$, $O(I),R(I)$ are homotopy equivalent to finite complexes. All homology is taken with respect to a fixed field $k$.

\begin{dfn}
A distinguished $(n,d)$-class of a module $R$ over an operad $O$ with a distinguished $n$-class $\alpha_I$ is a choice of elements $\beta_I$, for each nonempty finite set $I$, in \[H_{n|I|-n+d}(W(R)(I),\partial W(R)(I)),\] such that, under the isomorphism with $\bar{H}_*(B(R))$, the partial decomposites act as
\[\beta_{I \cup_a J} \rightarrow \beta_I \otimes \alpha_J.\]
\end{dfn}
As a sanity check, note \[(n|I|-n+d) + (n|J|-n)=n(|I|+|J|-1)-n+d=n|I \cup_a J|-n+d.\]

\begin{dfn}
A module $R$ with a distinguished class $\beta$ over a Poincaré-Koszul operad $O$ is Poincaré-Koszul of dimension $(n,d)$ if $\beta_I$ makes \[(W(R)(I),\partial W(R)(I))\] into a Poincaré duality pair for all $I$. We refer to $\beta$ as the fundamental class of the module.
\end{dfn}

\begin{dfn}
The module $S_{(n,d)}$ over $S_n$ is defined by \[S_{(n,d)}(I):=k[n|I|-n+d]\] with partial composites determined by the canonical isomorphism $k \otimes k \cong k$.
\end{dfn}

\noindent We define the $(n,d)$-suspension of a module $R$ over the operad $O$ in $(\mathrm{dgVect}_k,\otimes)$ by \[s_{(n,d)} R(I):= S_{(n,d)}(I) \otimes R(I).\] This naturally forms a module over $s_n O$. As such, a suspension of a module is a ``linear combination'' of two natural notions of suspension: one which is internal to $O$ modules and one which transforms $O$ modules to $s_n O$ modules.

By unraveling definitions one immediately obtains:

\begin{thm}
If $R$ is a Poincaré-Koszul module of dimension $(n,d)$ there is an isomorphism

\[H_*(R) \cong s_{(n,d)} H_*(K(R)).\]
induced by $H^*(R) \xrightarrow{\cap \beta} s_{(-n,-d)}\bar{H}_*(B(R))$, where we use the convention that $H^i(-)$ lives in degree $-i$. This is compatible with the Poincaré-Koszul isomorphism of $O$. This isomorphism is natural with respect to maps of Poincaré-Koszul modules of dimension $n$ that preserve the fundamental $(n,d)$-class

\end{thm}

An analogous version of Proposition \ref{prp:extra} holds for Poincaré-Koszul modules.

\begin{dfn}
A manifold module $R$ over a manifold operad $O$ is an open module where all the spaces are tame topological manifolds without boundary. We say $R$ is oriented (with respect to a field $k$) if $O$ is oriented, and we are supplied with local orientations of the $R(I)$ that make the partial composites orientation preserving.
\end{dfn}

\begin{dfn}
The dimension of a manifold module $R$ over the manifold operad $O$ is $(\operatorname{dim}(O),\operatorname{dim}(R(\{*\}))$.
\end{dfn}

For example, any manifold operad $O$ is a dimension $(\operatorname{dim}(O),0)$ module over itself since our operads are reduced. Given any manifold module $R$ of dimension $(n,d))$, the symmetric sequence $R \times \mathbb{R}^k$ is naturally a manifold module of dimension $(n,d+k)$.

As one expects, there is a corresponding Poincaré duality theorem for manifold modules:

\begin{thm}
Let $R$ be an oriented manifold module of dimension $(n,d)$ over an oriented manifold operad $O$. Then the homological Poincaré duality isomorphism, \[H_*(R)\cong s_{(n,d)}H_*((R^+)^\vee)\] is an isomorphism of modules compatible with the isomorphism \[H_*(O)  \cong s_nH_*((O^+)^\vee)\]
\end{thm}

\begin{thm}[Poincaré-Koszul duality for $E_M$]
If $M$ is a compact, framed $n$-manifold without boundary, $E_M$ is Poincaré-Koszul of dimension $(n,n)$. If $M$ is tame but not necessarily compact, we still have an isomorphism of modules
\[H_*(E_M)\cong s_{(n,n)}H_*(K( E_{M^+}))\]
compatible with the isomorphism of operads
\[H_*(E_n)\cong s_{n}H_*(K(E_n)).\]

\end{thm}
\begin{proof}
It is well known that the Fulton-MacPherson compactifications $\mathcal{F}_{M}(I)$ of the configuration spaces of a framed manifold form a proper module over $\mathcal{F}_n$ \cite[Proposition 6.4]{markl_1999}, and that this module is equivalent to $E_{M}$, compatible with the zigzag of equivalences \cite{salvatore_1998} \[\mathcal{F}_n \xleftarrow{\simeq} W(E_n)\xrightarrow{\simeq} E_n.\] 
This name is somewhat of a misnomer in the case $M$ is not compact because the resulting space remains noncompact. Heuristically, the Fulton-Macpherson construction removes noncompactness that arises from considering configurations, but not from the noncompactness of $M$. Hence, we one point compactify $\mathcal{F}_M$ to obtain $\mathcal{F}_{M^+}:=(\mathcal{F}_M)^+$, which is a model of $E_{M^+}$ via the above zigzag of operad equivalences.

If one examines Salvatore's construction of the isomorphism $\mathcal{F}_n \cong W(\mathcal{F}_n)$ \cite{salvatore_2021}, with minor alterations it shows $\mathcal{F}_M \cong W(\mathcal{F}_M)$. Salvatore's construction takes as input collar neighborhoods $c_i$ of $\mathcal{F}_n([i])$ and inductively constructs the homeomorphism $\mathcal{F}_n([i+1])\cong W (\mathcal{F}_n)([i+1])$ by appealing to the fact that every point of $\partial \mathcal{F}_n([i+1])$ is uniquely associated to a tree with $i+1$ leaves with nonleaf vertices labeled by elements of $\mathcal{F}_n$. Salvatore uses this description to associate to a point in the collar of $\partial \mathcal{F}_n([i+1])$ a ``lower tree'' and an ``upper tree'' which allows him to inductively construct a map from $  \mathcal{F}_n([i+1]) \rightarrow W(\mathcal{F}_n)([i+1])$ which he shows is a homeomorphism.

A similar decomposition of points in $\partial\mathcal{F}_M([i+1])$ exists where we now label the root by an element of $\mathcal{F}_M$ and all other nonleaf edges are still labeled by $\mathcal{F}_n$, see \cite[Section 6]{markl_1999} and \cite[Theorem 4.4]{Sinha_2004}. Using the same degrafting mechanism, one degrafts the tree into a ``lower tree'' and an ``upper tree'' and the construction proceeds as before except the lower tree analysis now uses the inductive homeomorphisms $\mathcal{F}_M([j]) \cong W(\mathcal{F}_M)([j])$ for $j < i+1$ and the upper tree analysis uses the homeomorphisms $\mathcal{F}_n([j]) \cong W(\mathcal{F}_n)([j])$ for $j<i+1$. The proof that the inductively constructed map is a homeomorphism works equally well in this case, as it only uses the unique tree description of the boundary.

This homeomorphism implies, by a dimension check, $\mathring{W}(\mathcal{F}_M)$ is a $(n,n)$-manifold module over $\mathring{W}_{(0,\infty)}(\mathcal{F}_n)$ where tameness of each space follows from the tameness of $M$. Hence, by Proposition \ref{prp:barM} and Poincaré duality for modules, there is an isomorphism \[H_*(\mathcal{F}_M)\cong s_{(n,n)}H_*(K(\mathcal{F}_{M^+}))\]
 
 If $M$ is compact, this isomorphism is easily identified with the Poincaré-Koszul duality isomorphism for modules, just as it was for operads.
\end{proof}




We end the first half of this paper by using Poincaré-Koszul duality to explore the structure of $H_*(F(M,-))$. Recall that the operad $\mathrm{com}$ in $(\mathrm{Top}_*,\wedge)$ is $S^0$ in every degree with all partial composites homeomorphisms.

\begin{dfn}
For a pointed space $X$, the $\mathrm{com}$ module $X^{\wedge}$ in $(\mathrm{Top}_*,\wedge)$  is \[X^\wedge(I):=X^{\wedge I}\]
The partial composites are induced by the diagonal of $X$.
\end{dfn}

\begin{prp}
For a tame framed manifold $M$ there is an isomorphism of shifted Lie modules \[\mathrm{res}_{s_n \mathrm{lie}}(H_*(\mathcal{F}_M)) \cong s_{(n,n)}H_*(K((M^+)^{\wedge}))\]
\end{prp}

\begin{proof}
There is a map $\mathcal{F}_{M^+ }\rightarrow (M^+)^\wedge$ given by collapsing infinitesimal configurations. This is compatible with the map of operads $(\mathcal{F}_n)_+ \rightarrow \mathrm{com}$, so induces a map $B(\mathcal{F}_{M^+})\rightarrow B((M^+)^\wedge)$ between bar constructions over the operads $\mathcal{F}_n$ and $\mathrm{com}$, respectively. 
This is known to be an equivalence as a consequence of \cite[Proposition 2.5]{FTW} and comes down to observing \[(\mathcal{F}_M(I))^+/(\partial \mathcal{F}_M(I))^+ \cong (M^+)^{\wedge I}/(\Delta^{\mathrm{fat}}((M^+)^{\wedge I})),\] where $\Delta^{\mathrm{fat}}(-)$ denotes the subspace where some points coincide. Taking duals and applying (not necessarily compact) Poincaré-Koszul duality and the Koszul self duality of the morphisms $s_n\mathrm{lie}\rightarrow \mathrm{pois}_n \rightarrow \mathrm{com}$ (see Section \ref{sec:algduality}) yields the result since the map $K((M^+)^\wedge)\rightarrow K(\mathcal{F}_{M^+})$ is a map of $K(\mathrm{com})$ modules.
\end{proof}

\begin{cor}
For a tame framed manifold $M$, the $s_n \mathrm{lie}$ module $\mathrm{res}_{s_n \mathrm{lie}}(H_*(\mathcal{F}_M))$ is a homotopy invariant of $M^+$.
\end{cor}

The effect of restriction from the n-Poisson operad to the shifted Lie operad is to restrict action of forests of trees on $H_*(F(M,-))$ to forests with a single tree. Hence, if the full $\mathrm{pois}_n$ module structure is to distinguish homotopy equivalent, but not framed diffeomorphic manifolds, it must be based on how the action treats disjoint unions of trees.

\section{The layers of the embedding tower}

In the previous sections, we studied actions of operads on modules which \textit{increased} cardinality. This increase is a consequence of defining operads via functors out of the category of \textit{nonempty} finite sets. If we include the empty set, we must supply maps 
\begin{align*} 
O(I)\otimes O(\emptyset)\rightarrow O(I-\{a\})\\
R(I)\otimes O(\emptyset)\rightarrow R(I-\{a\})
\end{align*}
When we require this type of map, we call the operad unitary. By a reduced unitary operad, we mean one that the underlying symmetric sequence $S$ has $S(\{*\})=S(\emptyset)=1$, the monoidal unit. In this case, the above partial composite reduces to maps $O(I)\rightarrow O(I -\{a\})$ and $R(I)\rightarrow R(I -\{a\})$.

For example, $E_n$ can be made into a reduced unitary operad through the action of forgetting disks, and similarly for $E_M$. Under the homotopy equivalence $E_M \simeq F(M,-)$, these maps correspond to the projections. The techniques of the first half of this paper are not suitable to study this extended action. For instance, the W-construction does not have an obvious way to incorporate these types of projection maps. Nevertheless, the study of the nonunitary modules $E_M$ is essentially equivalent to the study of the unitary versions by a result of Lurie\cite[Proposition 5.5.2.13]{lurieHA}.

We fix two compact smooth manifolds $M,N$, possibly of different dimensions. We refer to the interiors by $\mathring{M},\mathring{N}$. It is a theorem of Weiss (which we will make precise) that these projections determine the difference between the approximations $P_i(F)$ and $P_{i-1}(F)$ for a presheaf $F$ on $\mathring{M}$. Our approach to understanding these projections is via explicit models of their Spanier--Whitehead duals.

Suppose we have a ``good'' presheaf $F:\mathrm{Open}(\mathring{M})^{\mathrm{op}} \rightarrow \mathrm{Top}_*$, i.e. it takes isotopy equivalences to weak equivalences and sends increasing unions to homotopy inverse limits. The theory of manifold calculus, introduced by Weiss\cite{weiss_1999}, associates to $F$ a tower of functors approximating $F$ which are built from the values of $F$ on the sets diffeomorphic to disjoint unions of $\mathbb{R}^n$, where $n=\mathrm{dim}(M)$.

\[\begin{tikzcd}
	&& {P_\infty(F)} \\
	&& \dots \\
	&& {P_2(F)} \\
	&& {P_1(F)} \\
	F && {P_0(F)}
	\arrow[from=5-1, to=5-3]
	\arrow[from=5-1, to=4-3]
	\arrow[from=5-1, to=3-3]
	\arrow[from=5-1, to=1-3]
	\arrow[from=1-3, to=2-3]
	\arrow[from=3-3, to=4-3]
	\arrow[from=2-3, to=3-3]
	\arrow[from=4-3, to=5-3]
\end{tikzcd}\]
In\cite{brito_weiss_2013}, Boavida de Brito and Weiss show that in the case $F$ is ``context free'', this tower agrees with the embedding calculus tower associated to to a presheaf on the topological category $\mathrm{Disk}_n$, or equivalently, to the corresponding right modules over the unitary version of $E^{\mathrm{fr}}_n$.

Recall that for a set $I$, an $I$-cube is a functor $S:2^I \rightarrow C$, from the powerset of $I$, ordered by inclusion, to a category $C$.  For a contravariant $I$-cube $S:(2^I )^{\mathrm{op}} \rightarrow \mathrm{Top}_*$, $\mathrm{totfiber}$ denotes the total homotopy fiber of the cube, i.e.  \[\mathrm{totfiber}(S):=\mathrm{hofiber}(S(I) \rightarrow  \mathrm{holim}(S|_{2^I-I})).\]

Weiss proves the following result concerning the layers of the embedding tower\cite[Example 4.1.8]{goodwillie_klein_weiss_2001}:

\begin{thm}[Weiss]
For a good presheaf $F$ on $\mathring{M}$, the $i$th layer of the tower for $F$, $D_i(F):= \mathrm{hofiber}(P_i(F) \rightarrow P_{i-1}(F))$, is equivalent to the space of sections of the bundle over $F(\mathring{M},i)/\Sigma_i$ with fiber over a set $\{x_i\}$ given by 
$\mathrm{totfiber}_{J \subset I}( F(\sqcup_{j \in J} B_{\epsilon}(x_j)))$,
where we require the section to agree with the distinguished section near the fat diagonal.
\end{thm}

Here $B_\epsilon(-)$ denotes a sufficiently small ball, and the distinguished section is the section constant at the basepoint of the fibers. In the case $F=\mathrm{Emb}(-,\mathring{N})$ \footnote{One considers $\mathrm{Emb}(-,\mathring{N})$ as a functor taking values in pointed spaces by fixing a distinguished embedding $\mathring{M}\rightarrow \mathring{N}$.}, the cube we take the homotopy limit of is equivalent to the contravariant $[i]$-cube of framed configurations and projections:
\[J \subset [i]\]
\[J \rightarrow F^{\mathrm{fr}}(\mathring{N},J)\]

A notable fact is that when $i>1$ this can be replaced with the cube of unframed configurations\cite[Theorem 9.2]{weiss_1999}. If $i=1$, the approximation coincides with $\mathrm{FImm}(\mathring{M},\mathring{N})$, the space of formal immersions.

For stable embedding calculus, it is convenient to instead study presheaves on $\mathrm{Open}_{\mathrm{tame}}(\mathring{M})$, those open subsets which are diffeomorphic to the interior of a compact manifold with boundary. When we are in this context, we call a presheaf good if it simply sends isotopy equivalences to weak equivalences. This relaxation is warranted because in the tame setting infinite increasing unions are eventually homotopically constant, i.e. eventually all the inclusions are homotopy equivalences. The calculus of these presheaves is identical to the good presheaves on $\mathrm{Open}(M)$\cite[Section 4.1]{goodwillie_klein_weiss_2001}.

\begin{dfn}

A functor $F:\mathrm{Open}_{\mathrm{tame}}(\mathring{M})^{\mathrm{op}}\rightarrow \mathrm{Top}_*$ is an $N$-stable embedding functor if it is of the form $G(\Sigma^\infty_+ \mathrm{Emb}(-,\mathring{N}))$ for some $G: \mathrm{Sp}\rightarrow \mathrm{Top}_*$ which preserves weak equivalences.
\end{dfn}

Any $N$-stable embedding functor is automatically good because $\mathrm{Emb}(-,N)$ is a good presheaf on $\mathrm{Open}_{\mathrm{tame}}(M)$. The goal of the second half of this paper is to prove:

\begin{thm}
The layers of an $N$-stable embedding functor are relative tangential homotopy invariants of $(N,\partial N)$,  i.e. invariants of homotopy equivalences of pairs $f:(N,\partial N) \rightarrow (N',\partial N')$ such that $f^*(TN')\cong TN$.
\end{thm}

\section{Duality for subsets of $\mathbb{R}^n$}
In\cite{dold_puppe_1980}, Dold-Puppe proved the following space level generalization of Alexander duality.
\begin{thm}[Dold-Puppe]
Let $X \subset \mathbb{R}^n$. Then the following pairing

\[
X_+ \wedge \mathrm{cone}((\mathbb{R}^n - X) \rightarrow \mathbb{R}^n) \rightarrow \mathrm{cone}((\mathbb{R}^n-0) \rightarrow \mathbb{R}^n) \simeq S^n
\]

\[
(x,y,t)\rightarrow (x-y,t)
\]

is a duality pairing if $X$ is a compact neighborhood retract. In other words, the adjoint $\Sigma^\infty_+ X \rightarrow F(\Sigma^\infty \mathrm{cone}((\mathbb{R}^n - X) \rightarrow \mathbb{R}^n),\Sigma^\infty S^n )$ is an equivalence if $X$ is compact and has a neighborhood that retracts to it.

\end{thm}

This pairing has the excellent property that it is natural with respect to inclusion, i.e. covariant in the first variable and contravariant in the second. This makes it a great tool to extend space level duality arguments to diagrams of spaces.

\begin{dfn}
The Dold-Puppe functor is the functor $\mathrm{DP}:\mathrm{Open}(\mathbb{R}^n)^{\mathrm{op}}\rightarrow \mathrm{Top_*}$ given by \[U \rightarrow \mathrm{cone}((\mathbb{R}^n - U) \rightarrow \mathbb{R}^n)\]
\end{dfn}

Though only proved in the case of compact neighborhood retracts, the pairing of Dold-Puppe is a duality pairing for more general subsets.

\begin{dfn}
A subset $U$ of $\mathbb{R}^n$ is behaved if the pairing

\[U_+ \wedge \mathrm{DP}(U)\rightarrow \mathrm{DP}(\mathbb{R}^n-0)\]
is a duality pairing.
\end{dfn}

Excision and the theorem of Dold-Puppe easily show:

\begin{prp} \label{prp:behaved}
If $U\subset \mathbb{R}^n$ is homeomorphic to the interior of a compact manifold with boundary, then $U$ is behaved.
\end{prp}

\section{Fibrations over fibrations and invariance of bundles over configuration spaces}\label{sec:fib}

In this section, we study the combinatorics of fibrations over fibrations where we allow relativity in the base space and the fibers of either fibration. As an application, one can quickly generalize the homotopy invariance of stabilized configuration spaces to bundles over configuration spaces.

\begin{dfn}
A relative fibration $ (E,e)\rightarrow X$ is a fibration $E \rightarrow X$ such that the restriction $e\rightarrow X$ is a fibration.
\end{dfn}

 If the underlying fibration of a relative fibration is a fiber bundle, we may also refer to it as a relative fiber bundle. We will use the notation $\hat\partial$ to indicate the subfibration of a relative fibration; we reserve the unadorned $\partial$ to refer to the boundary of a manifold. Notably, if $\xi \rightarrow M$ is a manifold bundle with manifold fibers and $\hat\partial \xi$ is the fiberwise boundary fibration, it only agrees with $\partial \xi$ if $M$ is a manifold without boundary.

\begin{dfn}
Given a relative fibration $ (E,\hat\partial E) \to X$ and $A \subset X$, define the relative Thom space \[(X,A)^E:= E \cup \mathrm{cone}(E|_A \cup \hat\partial E)\]
Similarly, 
the reduced relative Thom space \[(X,A)^{\bar{E}}:= E/(E|_A \cup \hat\partial E).\]
\end{dfn}

In this notation, the subfibration $\hat\partial E$ is implicit. Some basic homotopy theory tells us:

\begin{prp}
For a relative fibration $(E,\hat\partial E)  \rightarrow  X$ with distinguished $A \subset X$, if the inclusion $E|_A \cup \hat\partial E \rightarrow E$ is a cofibration, then there is an equivalence $(X,A)^E \simeq (X,A)^{\bar{E}}$.
\end{prp}

\begin{dfn}
A relative fibration over a fibration is a relative fibration \\$(E,\hat\partial E) \rightarrow \xi$  together with a relative fibration $(\xi,\hat\partial \xi) \rightarrow X$ and a subspace $A \subset X$.
\end{dfn}

We emphasize that the above definition is relative in three ways: the total space of $E$, the total space of $\xi$, and the space $X$. We will refer to such relative fibrations over fibrations by $(E,\hat\partial E)\rightarrow (\xi,\hat\partial \xi) \rightarrow (X,A)$, but be warned that projections are not maps of relative spaces, and $(E,\hat\partial E)$ is not the relative fiber of $(\xi,\hat\partial \xi)$. 

Let $I$ be a finite set.

\begin{dfn}
 Given a relative fibration $(\xi,\hat\partial \xi) \rightarrow X$, we define \[(\xi^I,\hat\partial (\xi^I)):= (\xi^I,\{(p_i) \in \xi^I| \exists p_i \in \hat\partial \xi\}).\]
 This is a fibration over $X^I$ in the obvious way.
\end{dfn}

Given a fibration over a fibration \[(E,\hat\partial E)\rightarrow (\xi,\hat\partial \xi) \rightarrow (X,A),\] we may take external products to obtain another relative fibration over a fibration
\[(E^i,\hat\partial (E^I))\rightarrow (\xi^I,\hat\partial( \xi^I)) \rightarrow (X^I,\{(x_i)| \exists x_i \in A\}).\]

\begin{dfn}
Given
$ (E,\hat\partial E) \to (\xi,\hat\partial \xi) \rightarrow (X,A)$ we define \[(\xi,\hat\partial \xi)^{E^I /\Delta^{\mathrm{fat}}}:=(\xi^I,\hat \partial (\xi^I) \cup \xi^I|_{\Delta^{\mathrm{fat}}(X^I) \cup \{(x_i)| \exists x_i \in A\}})^{E^I}\]

Similarly,

\[(\xi,\hat\partial \xi)^{\bar{E}^I /\Delta^{\mathrm{fat}}}:=(\xi^I,\hat\partial (\xi^I)  \cup \xi^I|_{\{(x_i)| \exists x_i \in A\}})^{\bar{E}^I} \cup \mathrm{cone}(E^I|_{\Delta^{\mathrm{fat}}(X^I)})\]

\end{dfn}

Colloquially, these constructions are killing the boundary fibration of $E^I$, $E^I$ restricted to the boundary of $\xi^I$, anything over the fat diagonal of $X$, and anything over a tuple which has an element in $A$. 

The construction $(\xi,\hat\partial \xi)^{\bar{E}^I /\Delta^{\mathrm{fat}}}$ is a reduced version of $(\xi,\hat\partial \xi)^{E^I /\Delta^{\mathrm{fat}}}$, and these constructions will be interchangeable when certain cofibrancy conditions are met. It should be noted a cone still appears in the definition of $ (\xi,\hat\partial \xi)^{\bar{E}^I /\Delta^{\mathrm{fat}}}$; this avoids potentially technical questions about whether certain fat diagonals include as cofibrations.

\begin{prp}\label{prp:cofib}
Given $(E,\hat\partial E) \rightarrow (\xi,\hat\partial \xi) \rightarrow (X,A)$  such that $E|_{\xi|_A \cup \hat\partial \xi} \cup \hat\partial E$ includes into $E$ as a cofibration, then $(\xi,\hat\partial\xi)^{E^I/\Delta^{\mathrm{fat}}}\simeq (\xi,\hat\partial\xi)^{\bar{E}^I/\Delta^{\mathrm{fat}}}$.
\end{prp}

These constructions are preserved under a suitable notion of equivalence (which can surely be relaxed):

\begin{prp}\label{prp:equiv}
Suppose we have a map of relative fibrations over fibrations:

\[\begin{tikzcd}
	{(E,\hat\partial E)} & {(E',\hat\partial E')} \\
	{(\xi,\hat\partial\xi)} & {(\xi',\hat\partial\xi')} \\
	{(X,A)} & {(X',A')}
	\arrow[from=1-1, to=1-2]
	\arrow[from=2-1, to=2-2]
	\arrow[from=3-1, to=3-2]
	\arrow[from=2-1, to=3-1]
	\arrow[from=1-1, to=2-1]
	\arrow[from=1-2, to=2-2]
	\arrow[from=2-2, to=3-2]
\end{tikzcd}\]
where the bottom most horizontal map is the inclusion of a relative deformation retract, the middle horizontal map is the inclusion of a relative deformation retract for which the deformation covers the bottom deformation retraction, and the top most horizontal map is an inclusion covering the middle, and, when considered as an inclusion into $(E',\hat\partial E')|_{\xi}$, is a  fiberwise relative deformation retract which covers the identity of $\xi$.

Supposing this, the inclusion $(\xi,\hat\partial \xi)^{E^I /\Delta^{\mathrm{fat}}} \rightarrow (\xi',\hat\partial \xi')^{E'^I /\Delta^{\mathrm{fat}}}$ is a homotopy equivalence.

\end{prp}

\begin{proof}
What is necessary to check is that $(E',\hat\partial E')$ deformation retracts to $(E,\hat \partial E)$ in a manner that preserves all the fiber and relativity conditions. To construct such a deformation retraction, we start by considering the given deformation retraction $H:(\xi',\hat\partial \xi') \times I \rightarrow (\xi',\hat \partial\xi')$. By choosing a relative isomorphism \[(H^*(E'),H^*(\hat\partial E'))\cong (E' \times [0,1],\hat\partial  E' \times [0,1]),\] we may construct a fiberwise relative deformation retraction of $(E',\hat\partial E')$ to $(E',\hat\partial E')|_{\xi}$ which covers the deformation retraction of $\xi'$ to $\xi$. Composing with the given deformation retraction of $(E',\hat\partial E')|_{\xi}$ to $(E,\hat \partial E)$ finishes our argument.
\end{proof}

\begin{dfn}
    Given relative fibrations  $(\xi_i,\hat\partial \xi_i) \rightarrow (X_i,A_i)$ a relative fiberwise map is a pair of relative maps
    \[f:(X_1,A_1) \rightarrow (X_2,A_2)\]
    \[g:(\xi_1,\hat\partial\xi_1) \rightarrow (\xi_2,\hat\partial\xi_2)\]
such that $g$ is a lift of $f$.
\end{dfn}

\label{def:fibeq}

\begin{dfn}
Given relative fibrations $(\xi_i,\hat\partial \xi_i) \rightarrow (X_i,A_i)$, a relative fiberwise homotopy equivalence $(\xi_1,\hat\partial \xi_1)\rightarrow (\xi_2,\hat\partial \xi_2)$ is a relative fiberwise map which admits an inverse up to homotopy through relative fiberwise maps.
\end{dfn}
In order to get used to some of the combinatorics of relative fibrations over fibrations, we first state a relaxed version of Proposition \ref{prp:duality} which follows from an analog of \cite[Lemma 5.3]{malin} in the setting of relative fibrations over fibrations. Proposition \ref{prp:duality} will ultimately be proven in a more technical way by analyzing mapping cylinders of fiberwise homotopy equivalences. If $\xi$ is a bundle with smooth manifold fibers, we we will always let $\hat\partial\xi$ denote the bundle of fiberwise boundaries.

\begin{prp} \label{prp:weakinv}
Let $(\xi, \hat\partial\xi) \rightarrow (M,\partial M)$ be a smooth, relative fiber bundle. Fix a boundary preserving embedding of the manifold $\xi$ into $\mathbb{R}^{N-1} \times [0,\infty)$ with tubular neighborhood $\mu$. There is an equivariant equivalence \[\Sigma^\infty_+ (\xi^k|_{F(\mathring{M},k)}) \simeq \Sigma^{Nk} ((\Sigma^\infty(\xi,\hat\partial \xi )^{\mu^k /\Delta^{\mathrm{fat}}})^\vee),\] and thus, $\Sigma^\infty_+ (\xi^k|_{F(\mathring{M},k)})$ is a relative fiberwise homotopy invariant of the relative fiber bundle $(\xi,\hat\partial \xi)\rightarrow (M,\partial M)$.
\end{prp}

A similar result appeared in work of Moriya  which additionally takes into account action of the $E_1$ operad on the left hand side \cite[Theorem 1.1]{moriya}.

The righthand side of this equivalence certainly looks complicated, but it arises from simply keeping track of the boundary of the restriction of $\mu^k$ to $\xi^k|_{F(\mathring{M},k)}$ as a subset of $\mathbb{R}^N$ and applying Dold-Puppe duality. Before extending this result to include projections, we give two example applications of this proposition.

First, we can recover the main theorem of\cite{malin}, that $\Sigma^\infty_+ F(\mathring{M},k)$ is a proper homotopy invariant, by taking $\xi=M \xrightarrow{\mathrm{id}} M$ and appealing to the following:

\begin{prp}
If $\mathring{M}_1$ is properly homotopy equivalent to $\mathring{M}_2$, then there is a relative equivalence \[(M_1,\partial M_1) \simeq (M_2,\partial M_2) .\]
\end{prp}

\begin{proof}
Let $A_i$ denote the space of paths $[0,\infty] \rightarrow M_i^+$ such that the preimage of $\infty$ is $\infty$; this is sometimes called the space of ends of $M$. Let $X_i$ denote $\operatorname{cyl}(A_i \xrightarrow{\mathrm{ev}_0} M_i)$. Evidently, there is an equivalence $(X_1,A_1) \simeq (X_2,A_2)$. Fixing homeomorphisms $M_i^+ \cong M_i \cup \operatorname{cone}(\partial M_i)$, we may include $\partial M_i \rightarrow A_i$ as the space of paths starting at $z \in \partial M_i$ and flowing up the cone lines. It is straightforward to see that this is a deformation retract. This implies $(M_1,\partial M_1) \simeq (M_2,\partial M_2)$.
\end{proof}

\begin{dfn}
Given smooth manifolds with boundary $M_1,M_2$ a relative tangential homotopy equivalence $f:M_1 \rightarrow M_2$ is a relative homotopy equivalence such that $f^*(TM_2)\cong TM_1$.
\end{dfn}

If we let $\xi$ be the orthonormal frame bundle and apply Proposition \ref{prp:weakinv}, we conclude:

\begin{cor} If $M_1$ is relatively tangentially homotopy equivalent to $M_2$, then \[\Sigma^\infty_+ F^{\mathrm{fr}}(\mathring{M}_1,k) \simeq \Sigma^\infty_+ F^{\mathrm{fr}}(\mathring{M}_2,k),\] where $F^{\mathrm{fr}}(-,k)$ denotes the configurations of $k$ points with a framing at each point. 
\end{cor}

\section{Spivak normal fibrations}\label{sec:poincare}
Let $M_1,M_2$ be compact smooth manifolds with boundary with a relative homotopy equivalence $f:(M_1,\partial M_1)\simeq (M_2,\partial M_2).$  Let $\xi_i \rightarrow M_i$ be two bundles with compact, smooth manifold fibers, and a relative fiberwise homotopy equivalence $F:(\xi_1,\hat \partial \xi_1) \rightarrow (\xi_2,\hat\partial \xi_2)$ covering $f$.

For convenience, we assume that $F$ is a map of simplicial complexes for some triangulation of the $\xi_i$. Our strategy to prove that the $I$-cubes \[J \rightarrow \Sigma^\infty_+ F^\xi(\mathring{M_i},J),\] 
for $J \subset I$, are equivalent will be to construct a zigzag of equivalences passing through a cube associated to the relative fiberwise equivalence $F:\xi_1 \simeq \xi_2$. In order to do this, we will have to recall some facts about regular neighborhoods of Poincaré duality pairs. The homotopical nature of such neighborhoods was first studied by Spivak\cite{spivak_1967}. 

\begin{dfn}
A pair $(X,A)$ is an $n$-dimensional Poincaré duality pair if there is an $\alpha \in H_n(X,A)$ so that 
\[H^*(X)\xrightarrow{\cap \alpha} H_{n-*}(X,A)\]
\[H^*(X,A)\xrightarrow{\cap \alpha} H_{n-*}(X)\]
\[H^*(A) \xrightarrow{\cap \partial \alpha} H_{n-*+1}(A)\]
are all isomorphisms. Here $\partial$ denotes the connecting homomorphism in homology for the pair $(X,A)$.
\end{dfn}

Let $r:\mathrm{cyl}(f) \rightarrow M_2$ be the collapse map. We define $(\xi_{\mathrm{cyl}},\hat\partial\xi_{\mathrm{cyl}}):=(r^*(\xi_2),r^*(\hat\partial \xi_2 ))$. Let \[\partial \xi_{\mathrm{cyl}}=\xi_{\mathrm{cyl}}|_{\mathrm{cyl}(f|_{\partial M_1})}\cup \xi_{\mathrm{cyl}}|_{M_1} \cup \xi_{\mathrm{cyl}}|_{M_2} \cup \hat\partial \xi_{\mathrm{cyl}}.\]
We observe that the pair $(\xi_{\mathrm{cyl}}, \partial \xi_{\mathrm{cyl}}) $ forms a Poincaré duality pair, since it is equivalent to $(\xi_i\times [0,1],\partial (\xi_i \times [0,1]))$ which is a manifold and its boundary. 

We fix an embedding $\xi_{\mathrm{cyl}} \rightarrow D^N \times [0,1)$, avoiding $\partial D^N \times (0,1)$,
which sends only $\partial \xi_{\mathrm{cyl}}$ into $D^N \times \{0\}$, and sends only $\partial (\xi_{\mathrm{cyl}}|_{M_1} \sqcup \xi_{\mathrm{cyl}}|_{M_2})$ into $\partial D^N\times \{0\}$. This embedding can be assumed to be simplicial via our assumption that $F$ is simplicial, so it has a regular neighborhood we call $p:(\mu_{\mathrm{cyl}},\hat\partial \mu_{\mathrm{cyl}})\rightarrow \xi_{\mathrm{cyl}}$.
Restricting our embedding to  $\xi_{\mathrm{cyl}}|_{M_i}$, yields embeddings into $D^N$ which send $\partial(\xi_{\mathrm{cyl}}|_{M_i})$ into $\partial D^n$. These have regular neighborhoods given by $p|_{\xi_{\mathrm{cyl(f)}}|_{M_i}}$. We call these $(\mu_i,\hat\partial \mu_i)$. Note that the $\mu_i$ are honest disk bundles via the tubular neighborhood theorem, but $\mu_{\mathrm{cyl}}$ is possibly not even a fibration.

Recall that a point inside $\mathrm{PathFib}(g: X \rightarrow Y)$ consists of a pair $(x,\gamma)$ of a point $x \in X$ and a path $\gamma$ in $Y$ such that $\gamma(0)=g(x)$. There is a fibration \[\mathrm{ev}^g:\mathrm{PathFib}(g: X \rightarrow Y)\rightarrow Y\] \[(x,\gamma) \rightarrow \gamma(1).\]

\begin{dfn}
If $(X,A) \rightarrow (D^d,\partial D^d)$ is an embedding of an $n$-dimensional Poincaré duality pair with closed regular neighborhood $(\mu,\hat\partial \mu)$, then the Spivak normal fibration is 
\[(\operatorname{PathFib}( \mu \rightarrow X),\operatorname{PathFib}(\hat\partial \mu \rightarrow X))\rightarrow (X,A).\]
\end{dfn}

It is a theorem of Spivak\cite{spivak_1967} that this relative fibration has a fiber homotopy equivalent to $(D^{d-n},S^{d-n-1}) $.

Let $(\nu_{\mathrm{cyl}},\hat\partial \nu_{\mathrm{cyl}})$ denote the Spivak normal pair for the regular neighborhood $\mu_\mathrm{cyl}$ of $(\xi_{\mathrm{cyl}},\partial \xi_{\mathrm{cyl}})$. By the construction of the path fibration, there are maps $\mu_i \rightarrow \nu_{\mathrm{cyl}}$ via inclusion of the constant paths.

\begin{prp}\label{prp:squaresquare}
The following diagram is a zigzag of equivalence of relative fibrations over fibrations in the sense of Proposition \ref{prp:equiv}:
\[\begin{tikzcd}
	{(\mu_1,\hat\partial \mu_1)} & {(\nu_{\mathrm{cyl}},\hat\partial\nu_{\mathrm{cyl}})} & {(\mu_2,\hat\partial \mu_2)} \\
	{(\xi_{\mathrm{cyl}}|_{M_1},\hat\partial\xi_{\mathrm{cyl}}|_{M_1})} & {(\xi_{\mathrm{cyl}},\hat\partial\xi_{\mathrm{cyl}})} & {(\xi_{\mathrm{cyl}}|_{M_2}},\hat\partial\xi_{\mathrm{cyl}}|_{M_2}) \\
	{(M_1,\partial M_1)} & {(\mathrm{cyl}(f),\mathrm{cyl}(f|_{\partial M_1}) \cup \partial M_2)} & {(M_2,\partial M_2)}
	\arrow["\simeq", from=2-1, to=2-2]
	\arrow["\simeq"', from=2-3, to=2-2]
	\arrow[from=1-1, to=2-1]
	\arrow[from=1-2, to=2-2]
	\arrow[from=1-3, to=2-3]
	\arrow["\simeq", from=1-1, to=1-2]
	\arrow["\simeq"', from=1-3, to=1-2]
	\arrow["\simeq", from=3-1, to=3-2]
	\arrow["\simeq"', from=3-3, to=3-2]
	\arrow[from=2-1, to=3-1]
	\arrow[from=2-2, to=3-2]
	\arrow[from=2-3, to=3-3]
\end{tikzcd}\]

\end{prp}

\begin{proof}
It is classical that the bottom is a zigzag of inclusions of relative deformation retracts. Standard arguments with fiber bundles will allow us to lift these relative deformation retracts to relative fiberwise deformation retracts of the center row. The top zigzag follows from four facts. First, the $(\mu_i,\hat\partial\mu_i)$ are defined as the restriction of $(\mu_{\mathrm{cyl}},\hat\partial \mu_{\mathrm{cyl}})$ to $\xi_{\mathrm{cyl(f)|_{M_i}}}$. Second, $(\nu_{\mathrm{cyl}},\hat\partial\nu_{\mathrm{cyl}})$ is defined as the path fibration replacement of $(\mu_{\mathrm{cyl}},\hat\partial\mu_{\mathrm{cyl}})$. Third, the $ (\mu_i,\hat\partial\mu_i)$ are relative fibrations (in particular disk bundles). Fourth, the path fibration replacement of a relative fibration deformation retracts (relatively and preserving the fibers) onto the constant paths.
 
\end{proof}

\begin{lem}\label{lem:cofib}
The relative fibrations of the previous lemma satisfy the cofibrancy hypothesis of Proposition \ref{prp:cofib}. Hence,

\[
(\xi_{\mathrm{cyl}}|_{M_i},\hat\partial\xi_{\mathrm{cyl}}|_{M_i})^{\bar\nu_i^I/\Delta^{\mathrm{fat}}} \simeq (\xi_{\mathrm{cyl}},\hat\partial\xi_{\mathrm{cyl}})^{\bar\nu_{\mathrm{cyl}}^I/\Delta^{\mathrm{fat}}}.
\]
\end{lem}

\begin{proof}
We demonstrate the cofibrancy hypothesis for the relative fibration over a fibration over the mapping cylinder; it is made difficult by the necessity of taking path space fibrations. The other two follow from an easier, more direct argument.

   By \cite[Theorem 2]{Strom_1966}, a closed inclusion $W\subset X$ is a cofibration, if and only if, there is an open subset $U \supset W$ with a deformation retraction $H:U \times [0,1] \rightarrow X$ onto $W$, as well as a continuous function $\phi:X \rightarrow[0,1]$ such that $\phi^{-1}(0)=W$ and $\phi(X-U)=1$. 

   Suppose we have a map $f:E \rightarrow X$ such that $f^{-1}(A) \rightarrow E$ is a cofibration for a subspace $A \subset X$. We wish to show that \[(\mathrm{ev}^f)^{-1}(A) \rightarrow \mathrm{PathFib}(f)\]
is a cofibration. Suppose $U \supset f^{-1}(A),H:U \times [0,1] \rightarrow E, \phi: E \rightarrow [0,1]$ is the information witnessing that $f^{-1}(A) \rightarrow E$ is a cofibration. We have that $(\mathrm{ev}^f)^{-1}(U)$ is an open set which deformation retracts to $(\mathrm{ev}^f)^{-1}(A)$ in $\mathrm{PathFib}(f)$ as follows: For a path $\gamma$ ending in $U$, let $\gamma'_t$ denote the concatenation of $\gamma$, scaled to have domain  $[0,1-.5t\phi(\gamma(1))]$, with the path $\gamma(1)$ follows under $H$ until time $t$, scaled to have domain $[1-.5t\phi(\gamma(1)),1]$. Then the deformation retraction from $(\mathrm{ev}^f)^{-1}(U)$ to $(\mathrm{ev}^f)^{-1}(A)$ is \[((x,\gamma),t) \rightarrow (x,\gamma'_t)\]
And similarly, we can define a lift of $\phi$ to $\mathrm{PathFib}(f)$ by
\[(x,\gamma) \rightarrow \phi(\gamma(1)).\]
By construction, these satisfy the requirements to show \[(\mathrm{ev}^f)^{-1}(A) \rightarrow \mathrm{PathFib}(f)\] is a cofibration. Hence, by the definition of \[(\nu_{\mathrm{cyl}},\hat \partial\nu_{\mathrm{cyl}}) \rightarrow (\xi_{\mathrm{cyl}},\hat\partial \xi_{\mathrm{cyl}}) \rightarrow (\mathrm{cyl}(f), \mathrm{cyl}(f|_{\partial M_1}) \cup \partial M_2),\] to demonstrate the necessary cofibrancy of Proposition \ref{prp:cofib} it suffices to show the corresponding result for regular neighborhoods. Explicitly, that \[\mu_{\mathrm{cyl}}|_{(\xi_{\mathrm{cyl}}|_{ (\mathrm{cyl}(f|_{\partial M_1}) \cup \partial M_2)} \cup \hat\partial \xi_{\mathrm{cyl}})} \cup \hat\partial \mu_\mathrm{cyl} \subset \mu_\mathrm{cyl}\] is a cofibration. One could directly argue from the simplicial hypothesis on $F$ that all of these subsets have regular neighborhoods which is enough to construct the necessary deformation retraction of \cite[Theorem 2]{Strom_1966} mentioned above. We find it simpler to appeal to the classification of finite dimensional ANR's as locally contractible spaces \cite[Page 240]{Borsuk1932}, and the fact that a closed subspace of an ANR is an ANR, if and only if, the inclusion is a cofibration \cite[Proposition A.6.7]{Fritsch_Piccinini_1990}. Since this inclusion is a closed inclusion of locally contractible, finite dimensional spaces, we conclude that it is a cofibration.

With the cofibrancy hypothesis demonstrated, we may now apply Proposition \ref{prp:cofib} to the diagram of Proposition \ref{prp:squaresquare} which yields the claimed equivalences.
\end{proof}








\section{Cubes of configuration spaces}

In order to understand the layers of the embedding tower, it is apparent that we need to study cubes of framed configuration spaces. In section \ref{sec:fib}, we sketched an argument using the Dold-Puppe functor $\mathrm{DP}$, that the stabilized framed configuration spaces were tangential homotopy invariants. However, the functor $\mathrm{DP}$ is natural only with respect to embeddings, so  we must find a cube of inclusions which is equivalent to the cube of configuration space projections.
In order to attack this, we make repeated use of the bundle and fibration constructions of the previous two sections.

Let $(\xi,\hat\partial\xi) \rightarrow (M,\partial M)$ be a relative smooth manifold bundle over a compact, smooth manifold $M$ with a fixed embedding $(\xi,\partial \xi) \rightarrow (D^N,S^{N-1})$ preserving the boundary. Fix a normal embedded disk bundle $\iota: \mu \rightarrow D^n$ for this embedding.
\begin{dfn}
The contravariant $I$-cube of $\xi$-framed configurations $F^{\xi}(M,2^I)$ is defined by \[K \subset J \]\[ (\xi^J)|_{F(M,J)} \xrightarrow{\mathrm{pr}_K} (\xi^K)|_{F(M,K)}. \]
\end{dfn}

\begin{dfn}
The contravariant $I$-cube of open $\xi$-framed configurations $\mathring{F}^{\xi}(M,2^I)$ is defined as the restriction of the cube of manifolds $F^{\xi}(M,2^I)$ to their interiors.

\end{dfn}

Clearly, the open versions obtained by restriction to the interiors are equivalent to the closed versions of these cubes.

We are to study the situation of Section \ref{sec:poincare}, where we have a relative fiberwise equivalence $F:\xi_1 \rightarrow \xi_2$ covering $f:(M_1,\partial M_1) \rightarrow (M_2,\partial M_2)$. Unfortunately, the mapping cylinder of $F$ is not obviously a fibration over $\mathrm{cyl}(f)$, so instead we will work with the bundle $\xi_{\mathrm{cyl}}$, defined in the previous section as the pullback of $\xi_2$ to $\mathrm{cyl}(f)$. Define $(\xi'_i, \hat\partial \xi'_i)$ as the restriction $(\xi_{\mathrm{cyl}}|_{M_i},\hat\partial\xi_{\mathrm{cyl}}|_{M_i})$ which appears in Proposition \ref{prp:equiv}.





\begin{prp}
There are equivalences of cubes 
\[F^{\xi_i}(M_i,2^I) \simeq F^{\xi'_i}(M_i,2^I).\]
\[\mathring{F}^{\xi_i}(M_i,2^I) \simeq \mathring{F}^{\xi'_i}(M_i,2^I).\]
\end{prp}

\begin{proof}
By the definition of $\xi_{\mathrm{cyl}}$, if $i=2$, $F^{\xi_i}(M_i,2^I) = F^{\xi'_i}(M_i,2^I)$. If $i=1$, it follows since $(\xi_1,\hat\partial \xi_1) \simeq (\xi_2,\hat\partial \xi_2)$ through $F$. The statement for open configurations is analogous.
\end{proof}

The following cube is the key to allowing us to apply Dold-Puppe duality. It models the homotopy type of the projection maps between configuration spaces by open inclusions.

\begin{dfn}\label{def:thick}
The contravariant cube $\mu^{2^I}$ is defined by \[K \subset J \] 
\[
\mu^J \times (D^N)^{I -J} \xrightarrow{\prod_{k \in K}\operatorname{id}_{\mu} \times \prod_{j \in J -K} \iota \times \prod_{i \in I-J} \operatorname{id}_{D^n}} \mu^K \times (D^N)^{I -K}
\]
\end{dfn}


\begin{dfn}
The contravariant cube $\mathring{\mu}^{2^I}_{\mathrm{conf}}$ is defined by the  restriction of $\mu^{2^I}(J)$ to \[(\mu^J - \hat\partial(\mu^J))|_{\mathring{F}^{\xi}(M,2^I)(J)} \times (D^N-S^{N-1})^{I-J}.\]
\end{dfn}

Note, this cube is actually a cube in $\mathrm{Open}(\mathbb{R}^{NI})$.

\begin{prp}
There is an equivalence of cubes
\[\mathring{F}^{\xi}(M,2^I)\simeq \mathring\mu^{2^I}_{\mathrm{conf}}.\]
\end{prp}

\begin{proof}
The former includes into the latter as the inclusion of a space into a disk bundle over that space, hence it is an equivalence.
\end{proof}

\begin{dfn}
The covariant cube $(\xi,\hat\partial  \xi)^{\mu^{2^I}/\Delta^{\mathrm{fat}}}$ is given by $\mathrm{DP}(\mathring{\mu}^{2^I}_{\mathrm{\mathrm{conf}}})$.
\end{dfn}
We justify this name by observing that, objectwise, there is an equivalence
\[\mathrm{DP}(\mu^{2^I}_{\mathrm{\mathrm{conf}}})(J) \simeq (\xi,\hat\partial \xi)^{\mu^J/\Delta^{\mathrm{fat}}} \wedge (S^N)^{\wedge I-J}\]
which comes from examining the boundary of $\mathring\mu^{2^I}_{\mathrm{conf}}(J)$ as a subspace of $\mathbb{R}^{NI}$.
\begin{prp} \label{prp:duality}

There is an equivalence of cubes \[\Sigma^\infty_+\mathring{\mu}^{2^I}_{\mathrm{conf}} \simeq\Sigma^{NI}(\Sigma^\infty (\xi,\hat\partial  \xi)^{\mu^{2^I}/\Delta^{\mathrm{fat}}} )^\vee).\] 

\end{prp}

\begin{proof}
This follows from the naturality of the Dold-Puppe pairing and the fact that $\mathring{\mu}^{2^I}_{\mathrm{conf}}(J)$ is behaved by Proposition \ref{prp:behaved} since it is a bundle over $F(\mathring{M},J)$, which is the interior of the Fulton-MacPherson compactification of $F(M,J)$.
\end{proof}

\begin{dfn}
The covariant cube $(\xi,\hat\partial \xi)^{\bar{\mu}^{2^I}/\Delta^{\mathrm{fat}}}$ at $J \subset I$ is given by collapsing the subspace of $(\xi,\hat\partial  \xi)^{\mu^{2^I}/\Delta^{\mathrm{fat}}}(J)$ in which a coordinate labeled by an element of $J$ is in $\partial \mu$ or a coordinate labeled by an element of $J -I$ is in $S^{N-1}$.
\end{dfn}
This name is justified because there is a homeomorphism
\[(\xi,\hat\partial \xi)^{\bar{\mu}^{2^I}/\Delta^{\mathrm{fat}}}(J) \cong (\xi,\hat\partial \xi)^{\bar{\mu}^J/\Delta^{\mathrm{fat}}} \wedge (S^N)^{I-J}.\]

By Lemma \ref{lem:cofib} and Proposition \ref{prp:cofib} we see:

\begin{prp}
The collapse map $(\xi,\hat\partial \xi)^{\mu^{2^I}/\Delta^{\mathrm{fat}}} \rightarrow (\xi,\hat\partial \xi)^{\bar{\mu}^{2^I}/\Delta^{\mathrm{fat}}}$ is an equivalence.

\end{prp}

At this point, Proposition \ref{prp:duality} allows us to conclude a point-set model of the Spanier--Whitehead dual of the cube of $\xi_i$-framed configurations is
\[\Sigma^\infty (\xi'_i,\hat\partial \xi'_i)^{\bar{\mu}^{2^I}/\Delta^{\mathrm{fat}}}.\]
This means that there is a zigzag of equivalences of contravariant cubes \[\Sigma^\infty_+ F^{\xi_1}(M_1,2^I) \simeq \dots \simeq \Sigma^\infty_+ F^{\xi_2}(M_2,2^I),  \] if and only if, there is a zigzag of equivalences of covariant cubes \[\Sigma^\infty (\xi'_1,\hat\partial \xi'_1)^{\bar{\mu}_1^{2^I}/\Delta^{\mathrm{fat}}}\simeq  \dots \simeq \Sigma^\infty (\xi'_2,\hat\partial \xi'_2)^{\bar{\mu}_2^{2^I}/\Delta^{\mathrm{fat}}}.\]

Recall that $(\nu_{\mathrm{cyl}},\hat\partial \nu_{\mathrm{cyl}})$ is the path fibration of the projection of a regular neighborhood $p:(\mu_{\mathrm{cyl}},\hat\partial \mu_{\mathrm{cyl}})\rightarrow \xi_{\mathrm{cyl}}$. In other words, the points consist of pairs $(x,\gamma)$ of a point $x \in \mu_{\mathrm{cyl}}$ and a path $\gamma:[0,1]\rightarrow \xi_{\mathrm{cyl}(f)}$ such that $\gamma(0)=p(x)$. The projection to $\xi_{\mathrm{cyl}(f)}$ is defined to be $\gamma(1)$.

 We reiterate that for a relative fibration over a fibration \[(E,\hat\partial E)\rightarrow (\rho,\hat\partial \rho) \rightarrow (X,A)\] the notation $(\rho,\hat\partial \rho)^{\bar{E}^J/\Delta^{\mathrm{fat}}}$ makes both $\hat\partial E$ and $A$ implicit. Our next definition concerns the relative fibration over a fibration
\[(\nu_{\mathrm{cyl}},\hat\partial \nu_{\mathrm{cyl}})\rightarrow (\xi_{\mathrm{cyl}},\hat\partial\xi_{\mathrm{cyl}}) \rightarrow (\mathrm{cyl}(f),\mathrm{cyl}(f|_{\partial M_1}) \cup \partial M_2) \]

\begin{dfn}\label{def:cyl}
The covariant cube $(\xi_{\mathrm{cyl}},\hat\partial \xi_{\mathrm{cyl}})^{\bar\nu_{\mathrm{cyl}}^{2^I}/\Delta^{\mathrm{fat}}}$ is given by 

\[K \subset J:\]
\[
(\xi_{\mathrm{cyl}},\hat\partial\xi_{\mathrm{cyl}}))^{\bar\nu_{\mathrm{cyl}}^K/\Delta^{\mathrm{fat}}} \wedge (((D^N-S^{N-1}) \times [0,1])^+)^{I-K}\]\[ \downarrow \]\[(\xi_{\mathrm{cyl}},\hat\partial\xi_{\mathrm{cyl}}))^{\bar\nu_{\mathrm{cyl}}^J/\Delta^{\mathrm{fat}}} \wedge (((D^N-S^{N-1}) \times [0,1])^+)^{I-J} \]
We use the convention that elements not in the coned subspace have a cone coordinate of $0$. The function is given coordinate wise as follows:

\begin{itemize}

\item The function acts as the identity on the cone coordinate.

\item If $k \in K$, the $k$th coordinate of a point in the domain is represented by a point in $\nu_{\mathrm{cyl}}$ which consists of an element $x_k$ of $\mu_{\mathrm{cyl}}$ and a path $\gamma_k: [0,1] \rightarrow \xi_{\mathrm{cyl}}$ such that $\gamma_k(0)=p(x_k)$. The $k$th coordinate of the image will be the same element $x_k$ and the same path $\gamma_k$.

\item If $j \in J-K$, the $j$th coordinate of a point in the domain is represented by a point $x_j$ in $((D^N-S^{N-1})\times [0,1])^+$. If $x_j \in \mu_{\mathrm{cyl}}$, the $j$th coordinate of the image will be the pair of the element $x_j$ and the constant path at $p(x_j)$. Otherwise it will be the basepoint.

\item If $i \in I-J$, the $i$th coordinate is represented by a point $x_i$ in $(((D^N-S^{N-1}) \times [0,1])^+$, and the image will be $x_i$.

\end{itemize}

\end{dfn}

Continuity of the above construction is straightforward. There are two things to check: if $(\gamma_k(1))_{k \in K}$ projects to $\Delta^{\mathrm{fat}}(\mathrm{cyl}(f)^K)$, then the $J$-labeled coordinates of the image must project to $\Delta^{\mathrm{fat}}(\mathrm{cyl}(f)^J)$; and if $x_j$ for $j \in J-K$ is near $\partial \mu_{\mathrm{cyl}}$, then the image is near $* \in (\xi_{\mathrm{cyl}},\hat\partial\xi_{\mathrm{cyl}}))^{\bar\nu_{\mathrm{cyl}}^J/\Delta^{\mathrm{fat}}} \wedge (((D^N-S^{N-1}) \times [0,1])^+)^{I-J}$.

The first is immediate from the definition of the fat diagonal. The second is true because on these coordinates the map is the one point compactification of an open inclusion (together with the information of the constant path $p(x_j)$ if $j \in J-K$).

\begin{thm}
If  $\xi \rightarrow M$ is a fiber bundle with compact manifold fibers over a smooth manifold $M$,
the cube $\Sigma^\infty_+F^\xi(M,2^I)$ is a relative fiberwise homotopy invariant of $(\xi,\hat\partial\xi) \rightarrow (M,\partial M)$.
\end{thm}

\begin{proof}
 There is a zigzag of inclusions
\[\Sigma^\infty(\xi'_1,\hat\partial \xi'_1)^{\bar\mu_1^{2^I}/\Delta^{\mathrm{fat}}} \xrightarrow{\simeq}  
\Sigma^\infty(\xi_{\mathrm{cyl}},\hat\partial \xi_{\mathrm{cyl}})^{\bar\nu_{\mathrm{cyl}}^{2^I}/\Delta^{\mathrm{fat}}}
\xleftarrow{\simeq }\Sigma^\infty(\xi'_2,\hat\partial \xi'_2)^{\bar\mu_2^{2^I}/\Delta^{\mathrm{fat}}}\]
which are weak equivalences by Lemma \ref{lem:cofib}, so we are done.
\end{proof}


\begin{thm}
The layers of any $N$-stable embedding tower are a relative tangential homotopy invariant of $(N,\partial N)$.
\end{thm}

\begin{proof}
Suppose our functor is of the form $G( \Sigma^\infty_+ \mathrm{Emb}(-,\mathring{N}))$. The $k$th layer of the tower is classified by the bundle over $F(\mathring{M},k)/\Sigma_k$ with fibers over $\{x_i\}$ given by $\mathrm{totfib}(G( \Sigma^\infty_+ F^{\mathrm{fr}}(\mathring{N},2^{\{x_i\}}))$. By applying $G$ to the constructions of this section and letting $\xi$ be the frame bundle, we have that the cube 
\[J \subset \{x_i\}\]
\[J \rightarrow G(\Sigma^\infty_+ F^{\mathrm{fr}}(\mathring{N},J)),\] and thus its total homotopy fiber, are relative tangential homotopy invariants of $ (N,\partial N)$.

As the argument for the invariance of the cube depended only on a fixed embedding of $\xi_{\mathrm{cyl}}$, and nothing about the set $I$ itself, we can parametrize this over the cardinality $k$ subsets of $\mathring{M}$, and we deduce the classifying fibration is a relative tangential homotopy invariant which yields the result.

\end{proof}

\bibliographystyle{plain}
\bibliography{main.bib}







\end{document}